\documentclass[preprint,12pt]{elsarticle}

\usepackage{amssymb}
\usepackage{amsthm}
\usepackage{amsmath}
\usepackage{mathdots}
\usepackage{microtype}

\newtheorem{theorem}{Theorem}[section]
\newtheorem{lemma}[theorem]{Lemma}

\newtheorem{corollary}[theorem]{Corollary}

\theoremstyle{definition}
\newtheorem{definition}[theorem]{Definition}
\newtheorem{example}[theorem]{Example}

\newcommand{\abs}[1]{\left|#1\right|}

\journal{Linear Algebra and its Applications}

\begin{document}

\begin{frontmatter}



\title{$m$th roots of $H$-selfadjoint matrices}


\author[label1]{G.J.~Groenewald} 
\author[label1]{D.B.~Janse van Rensburg}
\author[label2]{A.C.M.~Ran}
\author[label1]{F.~Theron}
\author[label1,label3]{M.~van~Straaten}

\address[label1]{School~of~Mathematical~and~Statistical~Sciences,
North-West~University,
Research Focus: Pure and Applied Analytics,
Private~Bag~X6001,
Potchefstroom~2520,
South Africa.
E-mail: \texttt{gilbert.groenewald@nwu.ac.za, dawie.jansevanrensburg@nwu.ac.za, frieda.theron@nwu.ac.za, madelein.vanstraaten@nwu.ac.za}}
\address[label2]{Department of Mathematics, Faculty of Science, VU Amsterdam, De Boelelaan
    1111, 1081 HV Amsterdam, The Netherlands
    and Research Focus: Pure and Applied Analytics, North-West~University,
Potchefstroom,
South Africa. E-mail:
    \texttt{a.c.m.ran@vu.nl}}
\address[label3]{DSI-NRF Centre of Excellence in Mathematical and Statistical Sciences (CoE-MaSS)}

\begin{abstract}
In this paper necessary and sufficient conditions are given for the existence of an $H$-selfadjoint $m$th root of a given $H$-selfadjoint matrix. A construction is given of such an $H$-selfadjoint $m$th root when it does exist.
\end{abstract}




\end{frontmatter}


\section{Introduction and Preliminaries}
Let $H$ be an invertible $n\times n$ Hermitian matrix. On $\mathbb{C}^n$ we consider the indefinite inner product generated by $H$, given by $[x,y]=\langle Hx,y\rangle$, where $\langle \cdot\, ,\cdot \rangle$ denotes the standard inner product. Linear algebra in spaces with an indefinite inner product has been an area of active research over the past few decades, and many basic elements of the theory are summarized in \cite{GLR}. An $n\times n$ matrix $B$ is called $H$-selfadjoint if it is selfadjoint in the indefinite inner product given by $H$, or equivalently, if $HB=B^*H$.
The problem studied in this paper is that of finding $H$-selfadjoint $m$th roots of a given $H$-selfadjoint matrix $B$. This problem has been investigated in \cite{BMRRR} for $m=2$ where it plays a role in polar decompositions in an indefinite inner product space. Stability of $H$-selfadjoint square roots was studied in \cite{MRR}.

A matrix $A$ is called an $m$th root of a matrix $B$ if $A^m=B$. The problem of finding $m$th roots of a given matrix has been studied in the past; a first characterization can be found in the book by Wedderburn \cite{Wedderburn}, and another characterization in \cite{Psarrakos}. In the book by Gantmacher \cite{Gantm} it is shown for the singular case that the function of taking $m$th roots can be applied to each Jordan block.
Matrix $m$th roots have been studied extensively, see for example \cite{CrossLan,Higham,Smith,tenHave}.
%, where \cite{Smith} also discusses an algorithm for finding $m$th roots.

Obviously, in the case where $B$ is $H$-selfadjoint, a necessary condition for existence of an $H$-selfadjoint $m$th root is the existence of an $m$th root. Thus, this paper will focus on the extra conditions needed for the existence of an $H$-selfadjoint $m$th root. In \cite{MRR} existence and uniqueness of $H$-selfadjoint square roots of an $H$-selfadjoint matrix are studied, along with stability of such square roots when they exist. 
%In the problem of finding an $H$-polar decomposition of a matrix $X$, see \cite{BMRRR}, one is looking for an $H$-unitary $U$ and an $H$-selfadjoint $A$ such that $X=UA$. It is easy to see that a necessary condition for the existence of an $H$-polar decomposition of $X$ is that $H^{-1}X^*HX=A^2$, with the additional property that ${\rm Ker\, }X={\rm Ker\, }A$.  In particular, $H^{-1}X^*HX$ should have an $H$-selfadjoint square root $A$. 

More restrictive conditions on the $H$-selfadjoint square root can be imposed, for instance it is natural to impose the condition that the eigenvalues of $A$ are in the open right half-plane, possibly including zero as well. In more generality than in the present setting such polar decompositions and square roots have been studied extensively in a sequence of papers by Higham, Mackey, Mackey, Mehl and Tisseur \cite{HMMT,HMT,MMT}. It turns out that imposing this extra condition on the eigenvalues of the square root leads to a square root which is unique and is computable using iterative methods. However, it restricts the class of matrices for which such a square root exists to those for which the structure of the zero eigenvalue is semisimple (if one includes the possibility of zero being an eigenvalue, but insists on uniqueness) or to those which are non-singular (if one insists on the eigenvalues of the square root lying in the open right half-plane). 
Square roots for other classes of structured matrices have been considered as well in \cite{HMMT,HMT,MMT}; see also \cite{FMMX} for the case of Hamiltonian square roots of skew-Hamiltonian matrices.

\subsection{Notation}

The notation $\langle\cdot\,,\cdot\rangle$ stands for the standard inner product in either $\mathbb{C}^n$ or $\mathbb{R}^n$, i.e., 
$$
\langle x,y\rangle = \sum_{j=1}^{n} x_j\bar{y}_j,
$$
where $x=\begin{bmatrix}x_1 & \cdots & x_n\end{bmatrix}^T,\,\,y=\begin{bmatrix}y_1 & \cdots & y_n\end{bmatrix}^T \in\mathbb{C}^n$ or $\mathbb{R}^n$. The following definition and notation is taken from \cite{GLR}. A function $[\cdot\,,\cdot]$ from $\mathbb{C}^n \times \mathbb{C}^n$ to $\mathbb{C}$ is called an indefinite inner product in $\mathbb{C}^n$ if it is linear in the first argument, anti-symmetric and nondegenerate. 
Therefore, the only possible difference with the standard inner product is that $[x,x]$ may be nonpositive for $x\neq 0$. Clearly, for every $n\times n$ invertible Hermitian matrix $H$ (or real symmetric $H$) the formula $[x,y] = \langle Hx,y\rangle,\,\, x,y\in\mathbb{C}^n$, defines an indefinite inner product on $\mathbb{C}^n$. Conversely, for any indefinite inner product $[\cdot\,,\cdot]$ on $\mathbb{C}^n$, there exists an invertible Hermitian matrix $H$ such that $[x,y]=\langle Hx,y\rangle$ for all $x,y\in\mathbb{C}^n$. 

The $H$-adjoint of a square matrix $A$, denoted by $A^{[*]}$, is the unique square matrix such that $[Ax,y]=[x,A^{[*]}y]$ for all $x,y\in\mathbb{C}^n$. Observe that $A^{[*]} =H^{-1}A^*H$.

We denote a single $n\times n$ Jordan block with eigenvalue $\lambda\in\mathbb{C}$ by
$$
J_n(\lambda) = \begin{bmatrix}\lambda & 1  & 0 & \cdots & 0 \\0 & \lambda & 1 & \ddots & \vdots \\ \vdots & \ddots & \ddots & \ddots & 0 \\ 
\vdots & &\ddots& \lambda & 1 \\ 0 & 0 & \cdots & 0 & \lambda \end{bmatrix}.
$$
We will use the standard notation $\sigma(A)$ for the spectrum of $A$, i.e.\ for the set of eigenvalues of a matrix $A$, including nonreal eigenvalues of real matrices. Furthermore, we denote by $Q_n$ the $n\times n$ matrix with ones on the main anti-diagonal, which is called the backward identity matrix, or standard involutionary permutation (sip) matrix. 

We need the following well-known notation and result on Hermitian matrices. The inertia of a Hermitian matrix $H$ is a triple consisting of the number of positive, negative and zero eigenvalues, respectively, and will be denoted by $(i_+(H), i_-(H),i_0(H))$. According to Sylvester's law of inertia two Hermitian matrices have the same inertia if and only if they are congruent, see \cite{LanTis}.

\subsection{Important concepts}
A subspace $\mathcal{M}$ of $\mathbb{C}^n$ is called $H$\textit{-nondegenerate} if $x\in\mathcal{M}$ and $[x,y]=0$ for all $y\in\mathcal{M}$ implies that $x=0$. If $[x,y]=0$ for all $x,y\in\mathcal{M}$, then $\mathcal{M}$ is called $H$\textit{-neutral}.

A complex matrix $A$ is $H$\textit{-selfadjoint} if $A^{[*]}=A$, that is, if $HA = A^*H$. Thus, any $H$-selfadjoint matrix $A$ is similar to $A^*$. If we consider, for example, a single Jordan block $J_n(\lambda)$ with real eigenvalue $\lambda$ and if $Q_n$ is as defined above, then $J_n(\lambda)$ is $Q_n$-selfadjoint. Furthermore, the spectrum $\sigma(A)$ of an $H$-selfadjoint matrix $A$ is symmetric relative to the real axis. Also, the sizes of the Jordan blocks in the Jordan normal form of $A$ with eigenvalue $\lambda$ are equal to the sizes of the Jordan blocks with eigenvalue $\bar{\lambda}$. A proof of this result can be found in \cite[Proposition 4.2.3]{GLR}. 

A matrix $A$ is called $H$\textit{-unitary} if $A$ is invertible and $A^{[*]}=A^{-1}$, i.e., $A^*HA=H$. The pairs $(A_1,H_1)$ and $(A_2,H_2)$ are said to be \textit{unitarily similar} (following the terminology of \cite{GLR}) if there exists an invertible matrix $S$ such that $A_2=S^{-1}A_1S$ and $H_2=S^*H_1S$. If $H=H_1=H_2$, then $S$ is $H$-unitary and we say that $A_1$ and $A_2$ are $H$\textit{-unitarily similar}. Note that if $(A_1,H_1)$ and $(A_2,H_2)$ are unitarily similar, it implies that $A_1$ is $H_1$-selfadjoint if and only if $A_2$ is $H_2$-selfadjoint.

If a matrix $A$ is $H$-selfadjoint, then any power $A^k$ of $A$ is also $H$-selfadjoint since if we use $HA=A^*H$ repeatedly we have 
\begin{equation*}
HA^k=(HA)A^{k-1}=A^*HA^{k-1}=(A^*)^2HA^{k-2}=\cdots=(A^*)^kH=(A^k)^*H,
\end{equation*}
for any positive integer $k$.

\subsection{Canonical form}

The following theorem for the canonical form of $H$-selfadjoint matrices is taken from \cite{BMRRR}.

\begin{theorem}\label{ThmCanonform}
Let $H$ be an invertible Hermitian $n\times n$ matrix over the field $\mathbb{C}$, and let $A$ be an $n\times n$ $H$-selfadjoint matrix over $\mathbb{C}$. Then there exists an invertible $n\times n$ matrix $S$ over $\mathbb{C}$ such that $S^{-1}AS$ and $S^*HS$ have the form
\begin{eqnarray}\label{canonicalC}
S^{-1}AS &=& J_{k_1}(\lambda_1) \oplus \cdots \oplus J_{k_{\alpha}}(\lambda_{\alpha})\nonumber \\
&\oplus& [J_{k_{\alpha+1}}(\lambda_{\alpha+1}) \oplus J_{k_{\alpha+1}}(\overline{\lambda}_{\alpha+1})] \oplus \cdots \oplus 
[J_{k_{\beta}}(\lambda_{\beta}) \oplus J_{k_{\beta}}(\overline{\lambda}_{\beta})],
\end{eqnarray}
where $\lambda_1,\hdots,\lambda_{\alpha}$ are real and $\lambda_{\alpha+1},\hdots,\lambda_{\beta}$ are nonreal with positive imaginary parts; and
%\begin{eqnarray}\label{canonicalR}
%S^{-1}AS &=& J_{k_1}(\lambda_1) \oplus \cdots \oplus J_{k_{\alpha}}(\lambda_{\alpha})\nonumber \\
%&\oplus& J_{2k_{\alpha+1}}(\lambda_{\alpha+1} \pm i\mu_{\alpha+1}) \oplus \cdots \oplus
%              J_{2k_{\beta}}(\lambda_{\beta} \pm i\mu_{\beta} )
%\end{eqnarray}
\begin{equation}\label{canonicalH}
S^*HS = \varepsilon_1 Q_{k_1} \oplus \cdots \oplus \varepsilon_{\alpha}Q_{k_{\alpha}} \oplus Q_{2k_{\alpha+1}} \oplus \cdots \oplus Q_{2k_{\beta}},
\end{equation}
where $\varepsilon_1,\hdots,\varepsilon_{\alpha}$ are $\pm 1$. For a given pair $(A,H)$, where $A$ is $H$-selfadjoint, the canonical form \eqref{canonicalC} and \eqref{canonicalH} is unique up to permutation of orthogonal components in \eqref{canonicalH} and the same simultaneous permutation of the corresponding blocks in \eqref{canonicalC}.
\end{theorem}

The theorem is well-known and can be traced back to Weierstrass and Kronecker, see e.g.\ Chapter 5 in \cite{GLR}, and the references given there.% and \cite.

The signs $\varepsilon_1,\ldots,\varepsilon_\alpha$ in \eqref{canonicalH} form the \emph{sign characteristic} of the pair $(A,H)$. Therefore, the sign characteristic consists of signs $+1$ or $-1$ attached to every partial multiplicity (equivalently, the size of a real Jordan block in the Jordan normal form) of $A$ corresponding to a real eigenvalue. If $\varepsilon_i=1$ (resp.\ $\varepsilon_i=-1$) for some $i$, we say $\varepsilon_iQ_{k_i}$ is a \textit{positive} (resp.\ \textit{negative}) \textit{block} in $S^*HS$.

\subsection{The main theorem}
\label{SecMAIN}
We need a definition and more notation before stating the main result. Recall the following definition from \cite{Shapiro}.
\begin{definition} 
Let $A$ be a square matrix with Jordan blocks $\bigoplus_{i=1}^r J_{n_i}(\lambda)$ at $\lambda$ in its Jordan normal form  and assume that $n_1\geq n_2\geq n_3\geq\ldots\geq n_r>0$. The \textit{Segre characteristic} of $A$ corresponding to the eigenvalue $\lambda$ is defined as the sequence
$n_1,n_2,n_3,\ldots,n_r,0,0,\ldots$.
\end{definition}

Throughout the paper this sequence will only be used when looking at the Jordan blocks associated with the eigenvalue zero. Let $n'_1,n'_2,\ldots,n'_{pm},0,\ldots$ be some reordering of the Segre characteristic of a nilpotent $H$-selfadjoint matrix $B$, where $p$ is the number of nonzero $m$-tuples in this reordering. The $k$th $m$-tuple is $$(n'_{m(k-1)+1},n'_{m(k-1)+2},\ldots,n'_{mk}).$$ 
%Let the number of $m$-tuples which contains at least one nonzero number be $p$. 
Now let 
\begin{equation}\label{Bset}
\mathcal{B}_{\nu}^{(k)}:=\{i\mid n'_i=\nu\,;\;m(k-1)+1\leq i\leq mk\},
\end{equation}
%There exist a $\rho_k$ such that 
%\[n'_{km+1}=n'_{km+2}=\cdots=n'_{km+\rho_k}=a_k+1,\]
%and
%\[n'_{km+\rho_k+1}=\cdots=n'_{(k+1)m}=a_k.\]
%If $\rho_k$ is even, then the number of $i$'s 
so that $\big|\mathcal{B}_{\nu}^{(k)}\big|$ represents the number of blocks in $H$ (or $A$) of size $\nu$ which corresponds to the $k$th $m$-tuple, i.e.\ $k=1,\ldots,p$ and $\nu$ can be any number in the Segre characteristic of $B$.

Let us partition the canonical form $(J,H_B)$ of $(B,H)$ as follows:
\[J=J_0\oplus J_1\oplus J_2\quad\textup{and}\quad H_B=H_0\oplus H_1\oplus H_2,\]
where $J_0$ is a direct sum of blocks of the form $J_{k_j}(0)$ for some $k_j$, $J_1$ is a direct sum of blocks of the form $J_{k_j}(\alpha_j)$ for some $k_j$ and $\alpha_j<0$, and $J_2$ is a matrix in Jordan normal form with all eigenvalues not in $(-\infty,0]$, and where the matrices $H_0$, $H_1$ and $H_2$ correspond to the matrices $J_0$, $J_1$ and $J_2$.

In the next section, the search for necessary and sufficient conditions for existence of an $H$-selfadjoint $m$th root is reduced to the treatment of the same problem for the pairs $(J_0,H_0)$, $(J_1,H_1)$ and $(J_2,H_2)$ separately. These cases are then studied, and the results are summarized in the following main theorem. In addition, for each of these cases, in the next section an explicit construction is given in the case that an $H$-selfadjoint $m$th root exists.
\begin{theorem}\label{ThmMainThm}
Let $B$ be an $H$-selfadjoint matrix. Then there exists an $H$-selfadjoint matrix $A$ such that $A^m=B$ if and only if the canonical form of $(B,H)$, given by $(J,H_B)$, has the following properties:
\begin{enumerate}
\item There exists a reordering, $n_1',n'_2,\ldots,n_{pm}',0,\ldots$, of the Segre characteristic corresponding to the zero eigenvalue of $B$ such that for all $k$ the $m$-tuple $(n_{m(k-1)+1}',\ldots,n_{mk}')$ is descending and the difference between $n_{m(k-1)+1}'$ and $n_{mk}'$ is at most one. 
\item For some reordering satisfying the first property, the number of positive blocks in $H_0$ of size $\nu$ is equal to $\sum_{k=1}^p \pi_{\nu,k}$ where
\begin{equation*}
\pi_{\nu,k}=\begin{cases}
\frac{1}{2}\left(\big|\mathcal{B}_\nu^{(k)}\big|\right) & \textit{if }\big|{\mathcal{B}_\nu^{(k)}}\big|\textit{ is even} \\
\frac{1}{2}\left(\big|{\mathcal{B}_\nu^{(k)}}\big|+\eta_k\right) & \textit{if }\big|{\mathcal{B}_\nu^{(k)}}\big|\textit{ is odd}
\end{cases},
\end{equation*}
%\;\textit{with }\eta_k\textit{ either equal to }1\textit{ or }-1.
%\[\frac{1}{2}\left[\sum_{k=1}^p\left(\#\left\{i\mid i\in\mathcal{B}_\nu^{(k)};\;\abs{\mathcal{B}_\nu^{(k)}}\textup{ is even}\right\}+\#\left\{i\mid i\in\mathcal{B}_\nu^{(k)};\;\left|\mathcal{B}_\nu^{(k)}\right|\textup{ is odd}\right\}+\eta_k\right)\right] ,\]
with $\eta_k$ either equal to $1$ or $-1$. 
\item The blocks in $J_1$ and $H_1$ for an even $m$, can be reordered in the following way
\begin{equation*}
J_1=\bigoplus_{j=1}^t\left(J_{k_j}(\alpha_j)\oplus J_{k_j}(\alpha_j)\right)
\quad\textit{and}\quad
H_1=\bigoplus_{j=1}^t\left(Q_{k_j}\oplus(-Q_{k_j})\right).
\end{equation*}
\end{enumerate}
\end{theorem}

%%%%%%%%%%%%%%%%%%%%%%%%%%%%%%%%%%%%%%%%%%%%%%%%%%%%%%%%
%%%%%%%%%%%%%%%       SECTION TWO        %%%%%%%%%%%%%%%
%%%%%%%%%%%%%%%%%%%%%%%%%%%%%%%%%%%%%%%%%%%%%%%%%%%%%%%%

\section{Existence of an $H$-selfadjoint $m$th root}
Given an $H$-selfadjoint matrix $B$, we are interested in finding $H$-selfadjoint $m$th roots of $B$, if they exist. It is well-documented (for example in \cite[p.\ 461]{HornJohn}) that for finding $m$th roots, one may limit oneself to finding the $m$th root of the Jordan normal form of any matrix. We show that the same is true for finding $H$-selfadjoint $m$th roots.

Consider the following result which implies that it is sufficient to start with the pair $(B,H)$ in canonical form. It shows how $m$th roots of matrices which are part of pairs in the same equivalent class are related.

\begin{lemma}\label{S-1ASlemma}
Let the pair $(X,H_X)$ be unitarily similar to the pair $(Y,H_Y)$ where $H_X$ and $H_Y$ are invertible Hermitian matrices, i.e.\ there exists an invertible matrix $P$ such that
\begin{equation}\label{S*HAS=H}
P^{-1}XP=Y, \quad\textit{and}\quad
P^*H_XP=H_Y.
\end{equation}
Let the matrix $\tilde{A}$ be an $H_X$-selfadjoint $m$th root of $X$. Then the matrix $A:=P^{-1}\tilde{A}P$ is an $H_Y$-selfadjoint $m$th root of $Y$.
\end{lemma}
\begin{proof}
Suppose that the equalities in \eqref{S*HAS=H} hold. Let the matrix $\tilde{A}$ be an $H_X$-selfadjoint matrix such that $\tilde{A}^m=X$. Then by writing $(P^{-1}\tilde{A}P)^m=P^{-1}\tilde{A}^mP=Y$, we obtain an $m$th root $A:=P^{-1}\tilde{A}P$ of $Y$. It follows that $(\tilde{A},H_X)$ and $(A,H_Y)$ are unitarily similar and therefore the $m$th root $A$ is $H_Y$-selfadjoint.
\end{proof}

The above lemma can be used in the proofs for the existence of an $H$-selfadjoint $m$th root of a given $B$, where we only need to construct a matrix $\tilde{A}$ whose $m$th power has Jordan normal form equal to $B$ and check that there exists an invertible matrix $P$ such that equations \eqref{S*HAS=H} hold. 

From the literature (for example \cite[p.\ 461]{HornJohn}) we also know that when finding $m$th roots of a matrix $B$ in Jordan normal form, we may consider blocks with each eigenvalue separately. The situation is a bit more complicated in the case of $H$-selfadjoint $m$th roots, because of the restriction placed by the canonical form as given in Theorem~\ref{ThmCanonform}. We therefore need to group the blocks in $B$ in pairs of complex conjugate eigenvalues, except in the case where the eigenvalue is real, in which case it may appear on its own, see also \cite{BMRRR} for the case of finding $H$-selfadjoint square roots. We discuss the existence of $H$-selfadjoint $m$th roots for the cases where $B$ has only positive real eigenvalues, only nonreal eigenvalues, only eigenvalue zero, and only negative eigenvalues separately in the next few subsections.

Note that although functional calculus can be used for the case where $B$ has neither negative nor zero eigenvalues, we prefer to give a detailed proof using only linear algebra techniques.

%%adding picture in pdf
%\begin{center}
%\includegraphics[height=4cm]{contour}
%\end{center}

\subsection{The case where $B$ has only positive eigenvalues}
\label{SecPosReal}
In this subsection we find conditions for the existence of an $H$-selfadjoint $m$th root of an $H$-selfadjoint matrix which has only positive real numbers as eigenvalues. We start with a discussion which leads to the main result of this subsection. 

Let $B=J_{n}\left(\lambda\right)$ and $H=\varepsilon Q_{n}$, where $\varepsilon=\pm 1$ and $\lambda$ is a positive real number. Let $\tilde{A}=J_{n}\left(\mu\right)$, where $\mu$ is the positive real $m$th root of $\lambda$. Then the Jordan normal form of $\tilde{A}^m$ is equal to $B$. Note that the matrix $\tilde{A}$ is $H_A$-selfadjoint where $H_A=\delta Q_{n}$, for $\delta=1$ or $\delta=-1$. Take $\delta=\varepsilon$. Next, we construct an invertible matrix $P$ such that the equations
\begin{equation}\label{eqP-1AmPx2}
P^{-1}\tilde{A}^mP=B\;\;\textup{and}\;\; P^*H_AP=H
\end{equation}
hold and then Lemma~\ref{S-1ASlemma} can be applied. Therefore we examine the structure of the matrix $P^*H_AP$ where the columns of $P$ form a Jordan basis for the matrix $\tilde{A}^m$. Recall that, see e.g.\ \cite[p.\ 594]{Meyer}, the matrix $P$ will be of the form 
\begin{equation}\label{eqPform}
P=\begin{bmatrix}
(\tilde{A}^m-\lambda I)^{n-1}y & \cdots & (\tilde{A}^m-\lambda I)y & y
\end{bmatrix},
\end{equation}
where $y=(y_1,\ldots,y_n)^T\in\textup{Ker}(\tilde{A}^m-\lambda I)^n$ but $y\notin\textup{Ker}(\tilde{A}^m-\lambda I)^{n-1}$. Note that in this case $\textup{Ker}(\tilde{A}^m-\lambda I)^n=\mathbb{C}^n$.

We write $p_i=(\tilde{A}^m-\lambda I)^{n-i}y$ and let the entries of the matrix $P^*H_AP$ be denoted by $\phi_{i,j}(y)$. Then, see for example the proof of Theorem~\ref{ThmCanonform} in \cite{GLR}, the matrix $P^*H_AP$ is an anti-lower triangular Hankel matrix and can be uniquely determined by using only the $n$ values in its last row.
Therefore $P^*H_AP=H$ if and only if 
\begin{equation}\label{eqns}
\phi_{j}(y)=\phi_{n,j}(y)=[p_j,p_n]=y^*H_A(\tilde{A}^m-\lambda I)^{n-j}y=\begin{cases}1 & \textup{if }j=1,\\
0 & \textup{if }j=2,\ldots,n.
\end{cases}
\end{equation}

Since the matrix $(\tilde{A}^m-\lambda I)$ is upper triangular with zeros on the diagonal, the matrix $(\tilde{A}^m-\lambda I)^{n-j}$ will have zeros in the first $n-j$ columns and in the last $n-j$ rows. This implies that the entry $\phi_j(y):=\phi_{n,j}(y)=y^*H_A(\tilde{A}^m-\lambda I)^{n-j}y$ for each $j=1,\ldots,n$, is an expression in the variables $y_n,\ldots,y_{n-j+1}$, that is,
\begin{eqnarray*}\label{eqphivariables}
\phi_1(y)&=& \phi_1(y_n)\nonumber \\\nonumber
\phi_2(y)&=& \phi_2(y_n,y_{n-1})\\
\phi_3(y)&=& \phi_3(y_n,y_{n-1},y_{n-2})\\
\vdots && \vdots \nonumber\\
\phi_n(y)&=& \phi_n(y_n,\ldots, y_1),\nonumber
\end{eqnarray*}
where each one is a sum of terms of the form $c\bar{y}_iy_j$, $c\in\mathbb{R}$, see Example~\ref{exAddonPos}.

Hence, using the first equation in \eqref{eqns} to find the $n$th entry in $y$ and the other equations in \eqref{eqns} to write each of the other entries in $y$ in terms of the $n$th entry, we construct the vector $y$ and consequently find a matrix $P$ using \eqref{eqPform} such that equations \eqref{eqP-1AmPx2} hold.  It now follows from Lemma~\ref{S-1ASlemma} that the matrix $A:=P^{-1}\tilde{A}P$ is an $H$-selfadjoint $m$th root of $B$. 

The procedure is illustrated in the following example.
\begin{example}\label{exAddonPos}
Let $B=J_3(\lambda)$ where $\lambda$ is a positive real number and let $\tilde{A}=J_3(\mu)$ where $\mu$ is the positive real $m$th root of $\lambda$.  The matrix $B$ is $H$-selfadjoint where $H=Q_3$ and the matrix $\tilde{A}$ is $H_A$-selfadjoint where $H_A=H=Q_3$. Then the Jordan form of 
\begin{equation*}
\tilde{A}^m=\begin{bmatrix}
\lambda & m\mu^{m-1} & \frac{1}{2}m(m-1)\mu^{m-2} \\
0 & \lambda & m\mu^{m-1} \\
0 & 0 & \lambda
\end{bmatrix}
\end{equation*}
is $B$ and therefore we know that the invertible matrix $P$ for which the  equality $P^{-1}\tilde{A}^mP=B$ is true will be of the form
\begin{eqnarray*}
P&=&\begin{bmatrix}
(\tilde{A}^m-\lambda I)^2y & (\tilde{A}^m-\lambda I)y & y
\end{bmatrix}\\&=&\begin{bmatrix}
(m\mu^{m-1})^2y_3 & m\mu^{m-1}y_2+\frac{1}{2}m(m-1)\mu^{m-2}y_3 & y_1 \\
0 & m\mu^{m-1}y_3 & y_2 \\
0 & 0 & y_3
\end{bmatrix},
\end{eqnarray*} where $y=(y_1,y_2,y_3)^T\in\textup{Ker }(\tilde{A}^m-\lambda I)^3=\mathbb{C}^3$ but $y\notin\textup{Ker }(\tilde{A}^m-\lambda I)^2$. Then by equating the entries in the third row, $P^*HP=H$ holds if and only if the following equations hold:
\begin{eqnarray*}
1=&\phi_1(y_3)&=(m\mu^{m-1})^2\bar{y}_3y_3, \\
0=&\phi_2(y_3,y_2)&=m\mu^{m-1}\bar{y}_2y_3+\frac{1}{2}m(m-1)\mu^{m-2}\bar{y}_3y_3+m\mu^{m-1}\bar{y}_3y_2, \\
0=&\phi_3(y_3,y_2,y_1)&=\bar{y}_1y_3+\bar{y}_2y_2+\bar{y}_3y_1.
\end{eqnarray*}
Assume that $y$ is real, then we can solve the first equation for $y_3$, choose the positive value and substitute into the second equation to find $y_2$, and lastly solve the third equation for $y_1$. Therefore, one solution to these equations is as follows:
$$
y_1=\frac{-(m-1)^2}{32m\mu^{m+1}};\quad y_2=\frac{-(m-1)}{4m\mu^m};\quad y_3=\frac{1}{m\mu^{m-1}}.\qed
$$
\end{example}

In the case where $B$ consists of more than one block, the construction may be applied to each block separately. We have thus proved the following theorem. 

\begin{theorem}
Let $B$ be an $H$-selfadjoint matrix with a spectrum consisting only of positive real numbers.
%a spectrum consisting only of positive real numbers. 
Then there exists an $H$-selfadjoint matrix $A$ such that $A^m=B$.
\end{theorem}

\subsection{The case where $B$ has only nonreal eigenvalues}
\label{SecNonreal}
In this subsection we give a proof for the existence of an $H$-selfadjoint $m$th root in the case where $B$ has only nonreal eigenvalues.
\begin{theorem}
Let $B$ be an $H$-selfadjoint matrix with a spectrum consisting only of nonreal numbers. Then there exists an $H$-selfadjoint matrix $A$ such that $A^m=B$.
\end{theorem}

\begin{proof}
Let $B$ be a $2n\times 2n$ $H$-selfadjoint matrix with nonreal eigenvalues. Assume that the pair $(B,H)$ is in canonical form and that $B$ has only one number and its complex conjugate as eigenvalues, each with a geometric multiplicity of one. Thus, with $\lambda$ being a nonreal number,
%\[B=\bigoplus_{j=1}^t\left(J_{k_j}\left(\alpha_j\right)\oplus J_{k_j}\left(\bar{\alpha_j}\right)\right).
%\]
\[B=J_{n}\big(\lambda\big)\oplus J_{n}\left(\bar{\lambda}\right)\;\;\textup{and}\;\; H=Q_{2n}.
\]
Let $\mu$ be any $m$th root of $\lambda$ and
%\[\tilde{A}=\bigoplus_{j=1}^t\left(J_{k_j}\left(\beta_j\right)\oplus J_{k_j}\left(\bar{\beta_j}\right)\right),\]
\[\tilde{A}=J_{n}\left(\mu\right)\oplus J_{n}\left(\bar{\mu}\right),\]
 then the Jordan normal form of $\tilde{A}^m$ is equal to $B$. Note that the matrix $\tilde{A}$ is $H_A$-selfadjoint where $H_A=Q_{2n}$. We once again construct a $2n\times 2n$ invertible matrix $P$ such that $P^{-1}\tilde{A}^mP=B\;\;\textup{and}\;\; P^*H_AP=H$ hold.
For the first equality to hold, the columns of $P$ have to form a Jordan basis for the matrix $\tilde{A}^m$, and therefore $P=P_1\oplus P_2$ where
\begin{equation*}
P_1=\begin{bmatrix}
\big((J_n(\mu))^m-\lambda I\big)^{n-1}y & \cdots & y
\end{bmatrix}\end{equation*}
and
\begin{equation*}
P_2=\begin{bmatrix}
\left((J_n(\bar{\mu}))^m-\bar{\lambda} I\right)^{n-1}z & \cdots & z
\end{bmatrix},
\end{equation*}
where $y\in\textup{Ker}\left((J_n(\mu))^m-\lambda I\right)^n=\mathbb{C}^n$ but $y\notin\textup{Ker}\left((J_n(\mu))^m-\lambda I\right)^{n-1}$, and $z\in\textup{Ker}\left((J_n(\bar{\mu}))^m-\bar{\lambda} I\right)^n=\mathbb{C}^n$ but $z\notin\textup{Ker}\left((J_n(\bar{\mu}))^m-\bar{\lambda} I\right)^{n-1}$. Take $z=\bar{y}$, i.e.\ $P_2=\overline{P_1}$. Then by a simple calculation one finds that $P^*H_AP=H$ if and only if $P_1^TQ_nP_1=Q_n$, and by following a similar process as in Section~\ref{SecPosReal} one can see that this is true if and only if 
\begin{equation*}
\phi_{j}(y)=\phi_{n,j}(y)=y^TQ_n((J_n(\mu))^m-\lambda I)^{n-j}y=\begin{cases}1 & \textup{if }j=1,\\
0 & \textup{if }j=2,\ldots,n.
\end{cases}
\end{equation*}
Note that, similarly to case in Section~\ref{SecPosReal}, each $\phi_j(y)$ is an expression in the variables $y_n,\ldots,y_{n-j+1}$, and a solution to these equations can be found by solving the first equation for $y_n$ and substituting back into the other equations. Therefore, there exists a solution to $P^*H_AP=H$ which also satisfies $P^{-1}\tilde{A}^mP=B$, and hence by Lemma~\ref{S-1ASlemma} the matrix  $A:=P^{-1}\tilde{A}P$ is an $H$-selfadjoint $m$th root of $B$. In the case where $B$ consists of more than one pair of blocks, the construction can be applied to each pair of blocks separately.
\end{proof}

\subsection{The case where $B$ has only eigenvalue zero}
\label{SecZero}
In this section we consider the case where $B$ has only zero in its spectrum, so that $B$ is nilpotent.
Obviously the $m$th roots of $B$ will then be nilpotent as well. 

The main theorem concerning $H$-selfadjoint $m$th roots in this case is given. The first property is necessary for the existence of an $m$th root in general and this is well-known (see for example \cite{HornJohn}).
\begin{theorem}\label{ThmZeroEig}
Let $B$ be a nilpotent $H$-selfadjoint matrix. Then there exists an $H$-selfadjoint matrix $A$ such that $A^m=B$ if and only if the canonical form of $(B,H)$, given by $(J,H_B)$, has the following properties:
\begin{enumerate}
\item There exists a reordering, $n_1',n'_2,n_3',\ldots,n_{pm}',0,\ldots$, of the Segre characteristic of $B$ such that for all $k$ the $m$-tuple $(n_{m(k-1)+1}',\ldots,n_{mk}')$ is descending and the difference between $n_{m(k-1)+1}'$ and $n_{mk}'$ is at most one. 
\item For some reordering satisfying the first property, the number of positive blocks in $H_B$ of size $\nu$ is equal to $\sum_{k=1}^p \pi_{\nu,k}$ where
\begin{equation}\label{+blocksincases}
\pi_{\nu,k}=\begin{cases}
\frac{1}{2}\left(\big|\mathcal{B}_\nu^{(k)}\big|\right) & \textit{if }\big|{\mathcal{B}_\nu^{(k)}}\big|\textit{ is even}, \\
\frac{1}{2}\left(\big|{\mathcal{B}_\nu^{(k)}}\big|+\eta_k\right) & \textit{if }\big|{\mathcal{B}_\nu^{(k)}}\big|\textit{ is odd},
\end{cases}
\end{equation}
%\;\textit{with }\eta_k\textit{ either equal to }1\textit{ or }-1.
%\[\frac{1}{2}\left[\sum_{k=1}^p\left(\#\left\{i\mid i\in\mathcal{B}_\nu^{(k)};\;\abs{\mathcal{B}_\nu^{(k)}}\textup{ is even}\right\}+\#\left\{i\mid i\in\mathcal{B}_\nu^{(k)};\;\left|\mathcal{B}_\nu^{(k)}\right|\textup{ is odd}\right\}+\eta_k\right)\right] ,\]
with $\eta_k$ either equal to $1$ or $-1$.
\end{enumerate}
\end{theorem}

The proof will follow after some results and two examples.

\begin{lemma} \label{LemA^m} 
Let the matrix $A$ be equal to $J_n(0)$. Then $A^m$ has Jordan normal form 
\begin{equation}\label{eqJordA^m}
\bigoplus_{i=1}^rJ_{a+1}(0)\oplus\bigoplus_{i=1}^{m-r}J_a(0),
\end{equation}
where $n=am+r$, for $a,r\in\mathbb{Z}$, $0<r\leq m$.
\end{lemma}

Note that if $m\geq n$, then $a=0$ and $r=n$, i.e.\ \eqref{eqJordA^m} is equal to $\bigoplus_{i=1}^nJ_1(0)$, which is the $n\times n$ zero matrix.
\begin{proof}
Let $A=J_n(0)$. If $m\geq n$, then $A^m=(J_n(0))^m=0=\bigoplus_{i=1}^n J_1(0)$. Let $m<n$, then by raising $A$ to the $m$th power, we have
\[A^m=\begin{bmatrix}
0&\cdots&0&1&&& \\
&\ddots&&\ddots&\ddots&&\\
&&\ddots&&\ddots&&1 \\
&&&\ddots&&&0 \\
&&&&\ddots&&\vdots \\
&&&&&&0
\end{bmatrix}, \]
where the one in the first row is in the $(m+1)$th column. It can easily be seen that the Jordan chains of $A^m$ are
\begin{eqnarray}\label{eqJorChainsCj}
C_1&=&\{e_i \mid i\equiv 1 \,(\textup{mod }m);\; i=1,\ldots,n\},\nonumber\\
C_2&=&\{e_i \mid i\equiv 2 \,(\textup{mod }m);\; i=1,\ldots,n\},\nonumber\\
\vdots\;&&\quad\vdots \\
C_{m-1}&=&\{e_i \mid i\equiv m-1 \,(\textup{mod }m);\; i=1,\ldots,n\},\nonumber\\
C_m&=&\{e_i \mid i\equiv 0 \,(\textup{mod }m);\; i=1,\ldots,n\}.\nonumber
\end{eqnarray}
Use the division algorithm to write $n=am+r$ where $a,r\in\mathbb{Z}$, $0<r\leq m$. Then the number of elements in each set $C_j$ is
\begin{equation}\label{Cj_sizes}\abs{C_j}=\begin{cases}
a+1 & \textup{for}\;\;1\leq j\leq r, \\
a   & \textup{for}\;\;r+1\leq j\leq m.
\end{cases}
\end{equation}
Let $S$ be the $n\times n$ invertible matrix with columns consisting of the vectors in the Jordan chains $C_1,\ldots,C_m$. Then since the lengths of Jordan chains coincide with the sizes of the corresponding Jordan blocks and by using \eqref{Cj_sizes} we have the Jordan normal form of $A^m$: 
\begin{equation*}
S^{-1}A^mS=\bigoplus_{j=1}^mJ_{\abs{C_j}}(0)=\bigoplus_{i=1}^rJ_{a+1}(0)\oplus\bigoplus_{i=1}^{m-r}J_a(0).\qedhere
\end{equation*}
\end{proof}

An interesting corollary that we obtain from this result, is the following.

\begin{corollary}
If the number of Jordan blocks at zero of a matrix $B$ is not divisible by $m$, and there are no $J_1(0)$ blocks, i.e., no entry in the Segre characteristic of $B$ corresponding to the zero eigenvalue is equal to one, then $B$ does not have an $m$th root.
\end{corollary}

The following result shows the relation between the canonical forms as will be illustrated in the examples.

\begin{lemma}
The pair $(A,H)$ has canonical form $(J_n(0),\eta Q_n)$, $\eta=\pm 1$, if and only if $(A^m,H)$ has canonical form 
\begin{equation*}
\left(\bigoplus_{i=1}^r J_{a+1}(0)\oplus\bigoplus_{i=1}^{m-r} J_a(0),\,\bigoplus_{i=1}^r \varepsilon_i Q_{a+1}\oplus\bigoplus_{i=1}^{m-r} \varepsilon_{r+i}Q_a \right),
\end{equation*}
where the signs are as follows: If $r$ (resp.\ $m-r$) is even, the number of signs $\varepsilon_i$, where $i=1,\ldots,r$ (resp.\ $i=r+1,\ldots,m$), which are equal to $\eta$ is $\frac{r}{2}$ (resp.\ $\frac{m-r}{2}$). If $r$ (resp.\ $m-r$) is odd, the number of signs $\varepsilon_i$, where $i=1,\ldots,r$ (resp.\ $i=r+1,\ldots,m$), which are equal to $\eta$ is $\frac{r+1}{2}$ (resp.\ $\frac{m-r+1}{2}$). In both cases the rest of the signs are equal to $-\eta$.
\end{lemma}

The next two examples illustrate much of the general case to be proved after the examples.

\begin{example}\label{exJorChainsChange0}
Let the matrix $$B=\bigoplus_{i=1}^3J_4(0)\oplus\bigoplus_{i=1}^7J_3(0)\oplus\bigoplus_{i=1}^2J_2(0)$$ 
be $H$-selfadjoint where
$$H=\bigoplus_{i=1}^3\varepsilon_iQ_4\oplus\bigoplus_{i=1}^7\varepsilon_{3+i}Q_3\oplus\bigoplus_{i=1}^2\varepsilon_{10+i}Q_2$$ 
for some $\varepsilon_j=\pm1$. $B$ has Segre characteristic $4,4,4,3,3,3,3,3,3,3,2,2,0,\ldots$. We wish to determine for which values of $\varepsilon_j$ the matrix $B$ would have an $H$-selfadjoint fourth root. Thus, in terms of notation of Theorem~\ref{ThmMainThm}, $n=37$ and $m=4$. For the purpose of this example, we just consider the following grouping of the Segre characteristic into $4$-tuples:
\begin{equation}\label{eqExSegreOrdering}
(4,4,4,3),(3,3,3,3),(3,3,2,2),(0,0,0,0),\ldots.
\end{equation}
Then $p=3$ with $p$ as in Theorem~\ref{ThmMainThm}. (The other possibilities for reordering the Segre characteristic are
$(4,4,3,3),(4,3,3,3),(3,3,2,2),(0,0,0,0),\ldots$ and $(4,4,4,3),(3,3,3,2),(3,3,3,2),(0,0,0,0),\ldots$.)

Any fourth root of $B$ will be similar to $A=\bigoplus_{j=1}^q J_{n_j}(0)$  for some integers $q$ and $n_j$, since it will also be nilpotent. If $n_j=4a_j+r_j$ for some $a_j,r_j\in\mathbb{Z}$, $0<r_j\leq 4$, then by Lemma~\ref{LemA^m} we have that $A^4$ has Jordan form
\begin{equation}\label{eqExA^4}
\bigoplus_{j=1}^q\left(\bigoplus_{i=1}^{r_j}J_{a_j+1}(0)\oplus\bigoplus_{i=1}^{4-r_j}J_{a_j}(0)\right).
\end{equation} 
We also know that $A^4$ is similar to $B$ and therefore we know the number of blocks of order $4$, $3$ and $2$ in \eqref{eqExA^4}. If we restrict ourselves to the ordering \eqref{eqExSegreOrdering}, it can easily be seen that $a_1=3$, $r_1=3$, $a_2=2$, $r_2=4$, $a_3=2$ and $r_3=2$, which then give the values $n_1=15$, $n_2=12$ and $n_3=10$ from the division algorithm. Compare the exercises 6.4.10-6.4.13 in \cite{HornJohn}. Hence,  using the ordering  \eqref{eqExSegreOrdering}, any fourth root of $B$  will have the form $A=J_{15}(0)\oplus J_{12}(0)\oplus J_{10}(0)$. By Theorem~\ref{ThmCanonform}, the pair $(A,H_A)$ is in canonical form, with $H_A=\eta_1Q_{15}\oplus\eta_2Q_{12}\oplus\eta_3Q_{10}$ for some $\eta_1=\pm 1$, $\eta_2=\pm 1$ and $\eta_3=\pm 1$.  We wish to find a fourth root (which is similar to $A$) that is $H$-selfadjoint and to this end we construct a matrix $P$ satisfying \eqref{S*HAS=H} where $X=A^4$, $Y=B$, $H_X=H_A$ and $H_Y=H$.
Write down the Jordan chains of the matrix $A^4$ by using the notation in \eqref{eqJorChainsCj} adapted for more blocks:
\begin{equation*}
\begin{split}
&C_1=\{e_1,e_5,e_9,e_{13}\},\;C_2=\{e_2, e_6, e_{10}, e_{14} \},\;C_3=\{e_3,e_7,e_{11},e_{15}\},\\&C_4=\{e_4,e_8,e_{12}\};\;
C_{16}=\{e_{16},e_{20},e_{24}\},\;C_{17}=\{e_{17},e_{21},e_{25}\},\\
&C_{18}=\{e_{18},e_{22},e_{26}\},\;C_{19}=\{e_{19},e_{23},e_{27}\};\;C_{28}=\{e_{28},e_{32},e_{36}\},\\&C_{29}=\{e_{29},e_{33},e_{37}\},\;C_{30}=\{e_{30},e_{34}\},\;C_{31}=\{e_{31},e_{35}\}.
\end{split}
\end{equation*}
Note that the matrix having these Jordan chains as columns does not satisfy the equations \eqref{S*HAS=H}, since the only Jordan chains spanning $H_A$-nondegenerate spaces are $C_2$ and $C_4$. All the other Jordan chains span $H_A$-neutral spaces. Therefore a change of basis is necessary on these Jordan chains to ensure that all Jordan chains in the new basis span $H_A$-nondegenerate spaces. This is done in the following way: let $P$ be the invertible $37\times 37$ matrix 
%\begin{equation*}
%P=\begin{bmatrix}
%C_1^{+} & C_2 & C_1^{-} & C_4 & C_{16}^{+} & C_{17}^{+} & C_{17}^{-} & C_{16}^{-} & C_{28}^{+} & C_{28}^{-} & C_{30}^{+} & C_{30}^{-} 
%\end{bmatrix}
%\end{equation*}
whose columns consist of the following (new) Jordan chains
\begin{align*} C_1^+&=\{(e_1+e_3)/\sqrt{2}, (e_5+e_7)/\sqrt{2}, (e_9+e_{11})/\sqrt{2}, (e_{13}+e_{15})/\sqrt{2}\},\\ C_2&=\{e_2, e_6, e_{10}, e_{14} \},\\ C_1^-&=\{(e_1-e_3)/\sqrt{2}, (e_5-e_7)/\sqrt{2}, (e_9-e_{11})/\sqrt{2}, (e_{13}-e_{15})/\sqrt{2}\},\\ C_4 &=\{ e_4, e_8, e_{12} \},\\ C_{16}^+ &=\{ (e_{16}+e_{19})/\sqrt{2}, (e_{20}+e_{23})/\sqrt{2}, (e_{24}+e_{27})/\sqrt{2}\},\\ C_{17}^+ &=\{(e_{17}+e_{18})/\sqrt{2}, (e_{21}+e_{22})/\sqrt{2}, (e_{25}+e_{26})/\sqrt{2}\},\\ C_{17}^- &=\{(e_{17}-e_{18})/\sqrt{2}, (e_{21}-e_{22})/\sqrt{2}, (e_{25}-e_{26})/\sqrt{2}\},\\ C_{16}^- &=\{(e_{16}-e_{19})/\sqrt{2}, (e_{20}-e_{23})/\sqrt{2}, (e_{24}-e_{27})/\sqrt{2}\},\\ C_{28}^+ &=\{(e_{28}+e_{29})/\sqrt{2}, (e_{32}+e_{33})/\sqrt{2}, (e_{36}+e_{37})/\sqrt{2}\},\\ C_{28}^- &=\{(e_{28}-e_{29})/\sqrt{2}, (e_{32}-e_{33})/\sqrt{2}, (e_{36}-e_{37})/\sqrt{2}\},\\ C_{30}^+ &=\{(e_{30}+e_{31})/\sqrt{2}, (e_{34}+e_{35})/\sqrt{2}\},\\ C_{30}^- &=\{(e_{30}-e_{31})/\sqrt{2}, (e_{34}-e_{35})/\sqrt{2}\}.
\end{align*}
Then $P^{-1}A^4P=B$ holds, and comparison of entries in $P^*H_AP$ and $H$ gives us the relationship between the signs of $\varepsilon_j$ for $j=1,\ldots,12$ and those of $\eta_j$ for $j=1,2,3$. In order to determine which combinations of $\varepsilon_j$ could rise to $H$-selfadjoint $B$ with $H$-selfadjoint fourth roots, we consider all eight combinations of $\eta_j$ and determine the possible values of $\varepsilon_j$ that would correspond to these $\eta_j$ for each of the blocks associated with the $4$-tuples in  \eqref{eqExSegreOrdering}. 

Consider Table~\ref{theonlytable} which gives the signs of the blocks $Q_{n_i'}$ in $H$ corresponding to each entry $n_i'$ in the Segre characteristic of $B$ by specifying $\eta_j$ from $H_A$.\medskip

\begin{table}[h]
\begin{footnotesize}
$\begin{array}{rrr|cccc|cccc|cccc}
\multicolumn{3}{c}{4\textup{-tuples}}&4&4&4&3&3&3&3&3&3&3&2&2\\\cline{1-15}
&\varepsilon_i&&\varepsilon_1&\varepsilon_2&\varepsilon_3&\varepsilon_4&\varepsilon_5&\varepsilon_6&\varepsilon_7&\varepsilon_8&\varepsilon_9&\varepsilon_{10}&\varepsilon_{11}&\varepsilon_{12}\\
\eta_1&\eta_2&\eta_3&+\eta_1&+\eta_1&-\eta_1&+\eta_1&+\eta_2&+\eta_2&-\eta_2&-\eta_2&+\eta_3&-\eta_3&+\eta_3&-\eta_3\\\cline{1-3}
1&1&1&+&+&-&+&+&+&-&-&+&-&+&-\\
1&1&-1&+&+&-&+&  +&+&-&-& -&+&-&+\\
1&-1&1& +&+&-&+& -&-&+&+& +&-&+&- \\
1&-1&-1&+&+&-&+& -&-&+&+& -&+&-&+ \\
-1&1&1& -&-&+&-& +&+&-&-& +&-&+&-\\
-1&1&-1&-&-&+&-& +&+&-&-& -&+&-&+\\
-1&-1&1&-&-&+&-& -&-&+&+& +&-&+&- \\
-1&-1&-1&-&-&+&-& -&-&+&+& -&+&-&+
\end{array}$
\caption{Signs of all $\varepsilon_i$ corresponding to each combination of $\eta_j$}\label{theonlytable}
\end{footnotesize}
\end{table}

To explain Table~\ref{theonlytable} we look at the first row which shows that if $\eta_1$, $\eta_2$ and $\eta_3$ are all equal to $+1$, then $\varepsilon_1=\varepsilon_2=\varepsilon_4=\varepsilon_5=\varepsilon_6=\varepsilon_9=\varepsilon_{11}=+1$ and $\varepsilon_3=\varepsilon_7=\varepsilon_8=\varepsilon_{10}=\varepsilon_{12}=-1$ and that gives $H=Q_4\oplus Q_4\oplus -Q_4\oplus Q_3\oplus Q_3\oplus Q_3\oplus -Q_3\oplus -Q_3\oplus Q_3\oplus -Q_3\oplus Q_2\oplus -Q_2$.

Furthermore, we can see that for the first four choices of $\eta_1,\eta_2$ and $\eta_3$, $H$ consists of two positive $Q_4$ blocks, four positive $Q_3$ blocks and one positive $Q_2$ block. For the last four choices of  $\eta_1,\eta_2$, $\eta_3$, $H$ consists of one positive $Q_4$ block, three positive $Q_3$ blocks and one positive $Q_2$ block.

Thus for the chosen ordering \eqref{eqExSegreOrdering} the $H$-selfadjoint matrix $B$ will have an $H$-selfadjoint fourth root only if the total number of positive $Q_4$ blocks in $H$ is one or two, the total number of positive $Q_3$ blocks in $H$ is three or four, and there is only one positive $Q_2$ block in $H$.
%from the formula
%\[\frac{1}{2}\left[(4+2)+(1\pm 1) \right].\]\qed

If the pair $(B,H)$ was given, and therefore $\varepsilon_j$ is known for all $j=1,\ldots,12$ where some permutations are allowed, Table~\ref{theonlytable} then gives all possible sets of signs $\eta_j$ for $H_A$ such that $(A,H_A)$ is the canonical form for any $H$-selfadjoint fourth root of $B$, associated with the ordering \eqref{eqExSegreOrdering}, if it exists. For example if $\varepsilon_j=1$ for all $j=1,\ldots,12$, then no set of signs $\eta_j$ for $H_A$ exist, i.e.\ there does not exist an $H$-selfadjoint fourth root of $B$ associated with this ordering.

We also illustrate the use of \eqref{Bset} in this example for the ordering \eqref{eqExSegreOrdering}: $$\mathcal{B}_\nu^{(k)}=\{i\mid n'_i=\nu;\;4k-3\leq i\leq 4k\}$$
where $k=1,2,3$ and $\nu=4,3,2$ (the sizes of the blocks in $H$). Then $|{\mathcal{B}_4^{(1)}}|=3$, $|{\mathcal{B}_4^{(2)}}|=|{\mathcal{B}_4^{(3)}}|=0$, $|{\mathcal{B}_3^{(1)}}|=1$, $|{\mathcal{B}_3^{(2)}}|=4$, $|{\mathcal{B}_3^{(3)}}|=2$, $|{\mathcal{B}_2^{(1)}}|=|{\mathcal{B}_2^{(2)}}|=0$, $|{\mathcal{B}_2^{(3)}}|=2$. \qed
\end{example}

The following example shows how the sign characteristic differs by using different reorderings of the Segre characteristic.
\begin{example}
Let $B=\bigoplus_{i=1}^6J_3(0)\oplus\bigoplus_{i=1}^6J_2(0)$ and be $H$-selfadjoint where $(B,H)$ is in canonical form. Then the Segre characteristic of $B$ is $3,3,3,3,3,3,2,2,2,2,2,2,0,\ldots$. We illustrate finding the signs of the blocks in $H$ for which an $H$-selfadjoint sixth root of $B$ exists. For this we consider all of the possible reorderings of the Segre characteristic such that for each $6$-tuple, the maximum difference between any two numbers is one. In terms of the notation of Theorem~\ref{ThmMainThm}, $n=30$ and $m=6$, and in all of the reorderings $p=2$.

We follow a similar process as the one in Example~\ref{exJorChainsChange0} to determine the canonical form for $(B,H)$ that is necessary for the existence of an $H$-selfadjoint sixth root of $B$ for each possible reordering of the Segre characteristic.
\begin{enumerate}
\item Reordering: $(3,3,3,3,3,3),(2,2,2,2,2,2),(0,0,0,0,0,0),\ldots$. Canonical form of the $H$-selfadjoint sixth roots of $B$: $(J_{18}(0)\oplus J_{12}(0),\,\eta_1Q_{18}\oplus\eta_2Q_{12})$. Then
\begin{align*} H&=\eta_1Q_3\oplus\eta_1Q_3\oplus\eta_1Q_3\oplus -\eta_1Q_3\oplus -\eta_1Q_3\oplus -\eta_1Q_3\\ &\oplus  \eta_2Q_2\oplus\eta_2Q_2\oplus\eta_2Q_2\oplus -\eta_2Q_2\oplus -\eta_2Q_2\oplus -\eta_2Q_2.
\end{align*}
Thus, for any choice of $\eta_1$ and $\eta_2$, the number of positive $Q_3$ blocks in $H$ and the number of positive $Q_2$ blocks in $H$ are both three.

\item Reordering: $(3,3,3,3,3,2),(3,2,2,2,2,2),(0,0,0,0,0,0),\ldots$. Canonical form of the $H$-selfadjoint sixth roots of $B$: $(J_{17}(0)\oplus J_{13}(0),\,\eta_1Q_{17}\oplus\eta_2Q_{13})$. Then 
\begin{align*} H&=\eta_1Q_3\oplus\eta_1Q_3\oplus\eta_1Q_3\oplus -\eta_1Q_3\oplus -\eta_1Q_3\oplus \eta_1Q_2\\ &\oplus  \eta_2Q_3\oplus\eta_2Q_2\oplus\eta_2Q_2\oplus \eta_2Q_2\oplus -\eta_2Q_2\oplus -\eta_2Q_2.
\end{align*}
Thus, for the choice $\eta_1=1$ and $\eta_2=1$, the number of positive $Q_3$ blocks in $H$ and the number of positive $Q_2$ blocks in $H$ are both four. For both the choices $\eta_1=1$, $\eta_2=-1$, and $\eta_1=-1$, $\eta_2=1$, the number of positive $Q_3$ blocks and the number of positive $Q_2$ blocks are both three. If both $\eta_1$ and $\eta_2$ are chosen as $-1$, then the number of positive $Q_3$ blocks and the number of positive $Q_2$ blocks are both two.

\item Reordering: $(3,3,3,3,2,2),(3,3,2,2,2,2),(0,0,0,0,0,0),\ldots$. Canonical form of the $H$-selfadjoint sixth roots of $B$: $(J_{16}(0)\oplus J_{14}(0),\,\eta_1Q_{16}\oplus\eta_2Q_{14})$. Then
\begin{align*} H&=\eta_1Q_3\oplus\eta_1Q_3\oplus -\eta_1Q_3\oplus -\eta_1Q_3\oplus \eta_1Q_2\oplus -\eta_1Q_2\\ &\oplus \eta_2Q_3\oplus -\eta_2Q_3\oplus\eta_2Q_2\oplus \eta_2Q_2\oplus -\eta_2Q_2\oplus -\eta_2Q_2.
\end{align*}
Thus, like with the first reordering, for any choice of $\eta_1$ and $\eta_2$, the number of positive $Q_3$ blocks in $H$ and the number of positive $Q_2$ blocks in $H$ are both three.

\item Reordering: $(3,3,3,2,2,2),(3,3,3,2,2,2),(0,0,0,0,0,0),\ldots$. Canonical form of the $H$-selfadjoint sixth roots of $B$: $(J_{15}(0)\oplus J_{15}(0),\,\eta_1Q_{15}\oplus\eta_2Q_{15})$. Then 
\begin{align*} H&=\eta_1Q_3\oplus\eta_1Q_3\oplus -\eta_1Q_3\oplus \eta_1Q_2\oplus \eta_1Q_2\oplus -\eta_1Q_2\\ &\oplus \eta_2Q_3\oplus \eta_2Q_3\oplus -\eta_2Q_3\oplus \eta_2Q_2\oplus \eta_2Q_2\oplus -\eta_2Q_2.
\end{align*}
Again, the total number of positive blocks in $H$ is the same as with the second reordering. Thus, for the choice $\eta_1=1$ and $\eta_2=1$, the number of positive $Q_3$ blocks in $H$ and the number of positive $Q_2$ blocks in $H$ are both four. For both the choices $\eta_1=1$, $\eta_2=-1$, and $\eta_1=-1$, $\eta_2=1$, the number of positive $Q_3$ blocks and the number of positive $Q_2$ blocks are both three. If both $\eta_1$ and $\eta_2$ are chosen as $-1$, then the number of positive $Q_3$ blocks and the number of positive $Q_2$ blocks are both two.
\end{enumerate}
Note that we have covered all of the possibilities for the sixth root $A$ of $B$ as well as for the corresponding matrix $H_{A}$. Therefore this example shows the only options of matrices $H$ that we can start with in canonical form $(B,H)$ from which we will be able to find $H$-selfadjoint sixth roots.\qed 
\end{example}

We are now ready to prove the theorem giving the conditions of the existence of an $H$-selfadjoint $m$th root of nilpotent matrices.

\begin{proof}[Proof of Theorem~\ref{ThmZeroEig}]
Let $B$ be a nilpotent $H$-selfadjoint matrix with Segre characteristic $n_1,n_2,\ldots,n_r,0,\ldots$ and assume there exists an $H$-selfadjoint matrix $A$ such that $A^m=B$. Let $A$ be similar to $\bigoplus_{k=1}^pJ_{t_k}(0)$ for some $t_k$, then from Lemma~\ref{LemA^m} we have that the Jordan normal form of $A^m$ is equal to
\begin{equation}\label{JvanB}
J=\bigoplus_{k=1}^p\left[\bigoplus_{i=1}^{r_k}J_{a_k+1}(0)\oplus\bigoplus_{i=1}^{m-r_k}J_{a_k}(0) \right],
\end{equation} 
where $t_k=a_km+r_k$, $a_k,r_k\in\mathbb{Z}$ and $0<r_k\leq m$ by using the division algorithm. Since $A^m=B$, this matrix $J$ is also the Jordan normal form of $B$ and therefore will have the same Segre characteristic as $B$, possibly reordered. From \eqref{JvanB} one can see that this reordering, say $n'_1,n'_2,\ldots,n_{pm}',0,\ldots$, has the property that in each $m$-tuple the difference between the highest and the lowest number is at most one.

From Theorem~\ref{ThmCanonform} the pair $\left(\bigoplus_{k=1}^pJ_{t_k}(0),\bigoplus_{k=1}^p\eta_kQ_{t_k}\right)$ is in canonical form for $\eta_k=\pm 1$. Let the blocks of $H_B$ in the canonical form $(J,H_B)$ after a permutation of blocks according to the reordering $n'_1,n'_2,\ldots,n_{pm}',0,\ldots$, be given by $\varepsilon_iQ_{n'_i}$ for $\varepsilon_i=\pm 1$. The conditions on these signs can be found by a change in Jordan basis. The Jordan chains of $A^m$ which correspond to different Jordan blocks of $A$, or equivalently, to different $m$-tuples in the above reordering, are considered separately. Among the Jordan chains of $A^m$ of a certain length, say $\nu$, which correspond to a single Jordan block of $A$, there will be at most one Jordan chain spanning an $H$-nondegenerate space, and that will happen when there is an odd number of Jordan chains of this length since the other Jordan chains of length $\nu$ which do not span $H$-nondegenerate spaces, occur in pairs. By making combinations with these Jordan chains in a similar way as illustrated in Example~\ref{exJorChainsChange0}, we obtain the desired change in Jordan basis. Compare also the proof of Theorem~4.4 in \cite{BMRRR}. Each pair of Jordan chains delivers opposite signs of blocks in $H_B$ and the blocks in $H_B$ corresponding to the $H$-nondegenerate spaces will have the same sign as that obtained from the $m$th root. Hence by using $\mathcal{B}_\nu^{(k)}$ as introduced in \eqref{Bset} we can determine the number of blocks in $H_B$ for each sign. If for some $k$ the number $\big|{\mathcal{B}_\nu^{(k)}}\big|$ is even, then both the number of $i\in\mathcal{B}_\nu^{(k)}$ such that $\varepsilon_i=\eta_k$ and the number of $i\in\mathcal{B}_\nu^{(k)}$ such that $\varepsilon_i=-\eta_k$ is equal to $\frac{1}{2}\big(\big|\mathcal{B}_\nu^{(k)}\big|\big)$. If for some $k$ the number $\big|\mathcal{B}_\nu^{(k)}\big|$ is odd, then the number of $i\in\mathcal{B}_\nu^{(k)}$ such that $\varepsilon_i=\eta_k$ is $\frac{1}{2}\big(\big|\mathcal{B}_\nu^{(k)}\big|+1\big)$, and the number of $i\in\mathcal{B}_\nu^{(k)}$ such that $\varepsilon_i=-\eta_k$ is $\frac{1}{2}\big(\big|\mathcal{B}_\nu^{(k)}\big|-1\big)$. Hence the number of positive blocks in $H_B$ of size $\nu$ is equal to
\[\sum_{k=1}^p \pi_{\nu,k},\]
where $\pi_{\nu,k}$ is given by \eqref{+blocksincases}.

Conversely, suppose that $(B,H)$ is in canonical form and that it satisfies Properties 1 and 2, and suppose the reordering of the Segre characteristic of matrix $B$ that satisfies the second property is given by $n'_1,n'_2,\ldots,n_{pm}',0,\ldots$. Let $\tilde{A}=\bigoplus_{k=1}^pJ_{t_k}(0)$ where $t_k=\sum_{i=1}^mn'_{m(k-1)+i}$. If we have by the division algorithm that $t_k=a_km+r_k$, $a_k,r_k\in\mathbb{Z}$, $0<r_k\leq m$, then from Lemma~\ref{LemA^m} the matrix $\tilde{A}^m=\bigoplus_{k=1}^p(J_{t_k}(0))^m$ has the Jordan normal form
\[\bigoplus_{k=1}^p\left[\bigoplus_{i=1}^{r_k}J_{a_k+1}(0)\oplus\bigoplus_{i=1}^{m-r_k}J_{a_k}(0) \right].\]
Thus the Segre characteristic of this matrix is 
\begin{equation}\label{a_kSegre}
 (a_1+1,\ldots,a_1+1,a_1\ldots,a_1),\ldots,(a_p+1,\ldots,a_p+1,a_p,\ldots,a_p),0\ldots.
\end{equation}
Since for all $k=1,\ldots,p$ we have that
\[\sum_{i=1}^{r_k}(a_k+1)+\sum_{i=1}^{m-r_k}a_k=t_k=\sum_{i=1}^mn'_{m(k-1)+i}\quad\textup{and}\quad n'_{m(k-1)+1}-n'_{mk}\leq 1,\]
it then follows that the Segre characteristic in \eqref{a_kSegre} is equal to the sequence $n'_1,n'_{2},\ldots,n_{pm}',0,\ldots$. This means that $B$ is also the Jordan normal form of $\tilde{A}^m$ since they have the same Segre characteristic, or reordering thereof. Note also that the matrix $\tilde{A}$ is $H_A$-selfadjoint where $H_A=\bigoplus_{k=1}^p\varepsilon_kQ_{t_k}$. If we let $\varepsilon_k=\eta_k$ for each $k$ from Property $2$ which was assumed for $(B,H)$, then there exists an invertible matrix $P$, formed in the same way as explained above, such that $P^{-1}\tilde{A}^mP=B$ and $P^*H_AP=H$. Therefore, by Lemma~\ref{S-1ASlemma}, the matrix $A:=P^{-1}\tilde{A}P$ is an $H$-selfadjoint $m$th root of $B$.
\end{proof}

\subsection{The case where $B$ has only negative eigenvalues}
\label{SecNegReal}
We now look at the case where the eigenvalues of $B$ are negative real numbers, firstly for the case where $m$ is even and secondly where $m$ is odd.

Consider the following example regarding negative eigenvalues.

\begin{example}
Let $H=\varepsilon Q_2$,  where $\varepsilon=\pm 1$. Suppose that the $H$-selfadjoint matrix $$B=J_2(-1)=\begin{bmatrix}
-1&1\\0 &-1
\end{bmatrix}$$ has a square root $A$ which is $H$-selfadjoint. %Then the canonical form of $(A,H)$ is $(J,\varepsilon Q_2)$ or $(J,\varepsilon Q_1\oplus \varepsilon Q_1)$, with $\varepsilon=\pm 1$.
We know that $\sigma(A)\subseteq\{i,-i\}$ since $i$ and $-i$ are both square roots of $-1$. According to Theorem~\ref{ThmCanonform} the Jordan normal form of $A$ should be 
$$J=\begin{bmatrix}
i&0\\0&-i
\end{bmatrix},$$ which is $Q_2$-selfadjoint. Let $A=S^{-1}JS$ for an invertible matrix $S$, then 
\[A^2=(S^{-1}JS)^2=S^{-1}\begin{bmatrix}
-1&0\\0& -1
\end{bmatrix}S, \]
which is not similar to $B$. This gives a contradiction. Therefore $B$ does not have a square root which is $H$-selfadjoint. Note however, that $\begin{bmatrix} i & -\frac{1}{2}i \\ 0 & i\end{bmatrix}^2=B$, so $B$ does in fact have a square root.\qed
\end{example}

This illustrates the fact that there do not exist $H$-selfadjoint $m$th roots of matrices of the form $J_n(\lambda)$ where $m$ is even and $\lambda$ is a negative real number. Suppose in general that $B=J_n(\lambda)$, with $\lambda$ a negative real eigenvalue. We know from Theorem~\ref{ThmCanonform} that $(B,\varepsilon Q_n)$ is in canonical form. But the $m$th roots of $\lambda$, where $m$ is even, are all nonreal numbers, since no real number raised to the $m$th power can be negative if $m$ is even. Therefore the eigenvalues of any $m$th root $A$ of $B$ are nonreal numbers. From Theorem~\ref{ThmCanonform} we know that the Jordan normal form of $A$ should contain pairs of Jordan blocks of equal size that correspond to complex conjugate pairs. This implies that the Jordan normal form of $A$ should consist of a direct sum of at least two blocks of equal size, but such a matrix raised to the $m$th power is not similar to $B$. Hence the blocks of the canonical form of $(B,H)$ corresponding to negative real eigenvalues should occur in pairs.

Now we present a lemma that will be useful in the proof of the subsequent theorem.

\begin{lemma}\label{LemOneBlockNegEig}
Let $T$ be an upper triangular complex $n\times n$ Toeplitz matrix with $\lambda\in\mathbb{R}$ on the main diagonal, and a nonzero number on the superdiagonal. Let 
$B=T\oplus \overline{T}$. Then $B$ is $Q_{2n}$-selfadjoint, and the pair $(B, Q_{2n})$ is unitarily similar to $(J_n(\lambda) \oplus J_n(\lambda), Q_n \oplus -Q_n)$. 
\end{lemma}

\begin{proof}
Let
\[
T=\begin{bmatrix} \lambda & t_2 & \cdots & t_ n \\ \ & \ddots & \ddots  & \vdots \\ \ & \ & \ddots & t_2 \\
\ & \ & \ & \lambda\end{bmatrix},
\]
with $t_2\not= 0$.
Then 
\begin{eqnarray*}
Q_{2n} B=&\begin{bmatrix} 0 & Q_{n}\\
Q_n & 0 \end{bmatrix}
\begin{bmatrix} T & 0 \\ 0 & \overline{T}\end{bmatrix} =
\begin{bmatrix} 0 & Q_n\overline{T} \\ Q_n T & 0 \end{bmatrix} \\
=& \begin{bmatrix} & \begin{matrix}
& & & \lambda \\ 
& & \iddots & \overline{t_2} \\
& \iddots & \iddots & \vdots \\
\lambda & \overline{t_2} & \cdots & \overline{t_n}
\end{matrix}
\\
\begin{matrix}
& & & {\lambda} \\ 
& & \iddots & {t_2} \\
& \iddots & \iddots & \vdots \\
{\lambda} & {t_2} & \cdots & {t_n}
\end{matrix} &
\end{bmatrix}
\end{eqnarray*}
which is clearly selfadjoint. Hence $B$ is $Q_{2n}$-selfadjoint. 

Because $t_1=\lambda$ is real and $t_2\not=0$ the Jordan normal form of both $T$ and $\overline{T}$ is $J_n(\lambda)$. Thus the canonical form of the pair $(B, Q_{2n})$ is given by $(J_n(\lambda)\oplus J_n(\lambda) , \varepsilon_1 Q_n \oplus \varepsilon_2 Q_n)$. Therefore there exists an invertible matrix $S$ such that $S^{-1}BS=J_n(\lambda)\oplus J_n(\lambda)$ and $S^*Q_{2n}S=\varepsilon_1Q_n\oplus\varepsilon_2Q_n$. By using these equations and letting $Q=\varepsilon_1Q_n\oplus\varepsilon_2Q_n$ we have the following congruence:
\begin{equation*}
S^*Q_{2n}BS=S^*Q_{2n}SS^{-1}BS=Q\,(J_n(\lambda)\oplus J_n(\lambda)),
\end{equation*}
which leads to
\begin{eqnarray*}
S^*Q_{2n}(B-\lambda I)S=Q\,(J_n(0)\oplus J_n(0)),
\end{eqnarray*}
and then by multiplying $(S^{-1}(B-\lambda I)S)^{n-2}$ from the right, we have
\begin{equation}\label{eqlem2.12}
S^*Q_{2n}(B-\lambda I)^{n-1}S=Q\,(J_n(0)\oplus J_n(0))^{n-1}.
\end{equation}
Note that the matrix $Q_{2n}(B-\lambda I)^{n-1}$  has only two nonzero entries: $t_2^{n-1}$ in the $(2n,n)$ position and its complex conjugate  in the $(n,2n)$ position. Consequently $Q_{2n}(B-\lambda I)^{n-1}$ has only one positive and one negative eigenvalue, namely $\pm(t_2\overline{t_2})^{\frac{n-1}{2}}$, and the matrix on the right side of \eqref{eqlem2.12} has eigenvalues $\varepsilon_1$ and $\varepsilon_2$. Hence by Sylvester's law of inertia  $\varepsilon_1=-\varepsilon_2$, and by reordering the Jordan blocks if necessary we have that the pair $(B, Q_{2n})$ is unitarily similar to 
$(J_n(\lambda)\oplus J_n(\lambda) , Q_n \oplus  -Q_n)$ as claimed.
\end{proof}

First we consider the case where $m$ is even.

\begin{theorem}\label{ThmNegEig}
Let $B$ be an $H$-selfadjoint matrix with a spectrum consisting of only negative real numbers. Then there exists an $H$-selfadjoint matrix $A$ such that $A^m=B$, for $m$ even, if and only if the canonical form of $(B,H)$, given by $(J,H_B)$, has the following form:
%\begin{enumerate}
\begin{equation} \label{Jnegeiew}
J=\bigoplus_{j=1}^t\left(J_{k_j}(\lambda_j)\oplus J_{k_j}(\lambda_j)\right), \quad \lambda_j\in\mathbb{R}^-
\end{equation} and 
\begin{equation} \label{Hnegeiew}
H_B=\bigoplus_{j=1}^t\left(Q_{k_j}\oplus (-Q_{k_j})\right).
\end{equation}
%\end{enumerate}
\end{theorem}
\begin{proof}
%We start with the following observation: let
%\[
%J=J_n(\lambda)\oplus J_n(\lambda) , \quad\textup{and}\quad H_B=Q_n\oplus (-Q_n),
%\]
%with $\lambda <0$.
%Take $P=\frac{1}{\sqrt{2}}\begin{bmatrix} I_n & I_n \\ I_n & -I_n \end{bmatrix}$. Then $P^*=P$, and $P^{-1}=\frac{1}{2}P$. Consider the pair $(P^{-1}JP,\, P^*H_BP)$. One computes that $P^{-1}JP=J$ and
%$P^*H_BP=Q_{2n}$. So, the pair $(J, Q_n\oplus -Q_n)$ is unitarily similar to the pair $(J, Q_{2n})$, 
%and so in \eqref{Hnegeiew} we may replace $Q_{k_j}\oplus-Q_{k_j}$ by $Q_{2k_j}$ in each of the $t$ summands.
%
%
Let $B$ be an $H$-selfadjoint matrix with a spectrum consisting of only negative real numbers and assume there exists an $H$-selfadjoint matrix $A$ such that $A^m=B$, with $m$ even. 
Then the eigenvalues of $A$, which are $m$th roots of negative real numbers $\lambda_j\in\sigma(B)$, will be nonreal numbers. Since $A$ is $H$-selfadjoint, the eigenvalues of $A$ must be symmetric relative to the real axis. Therefore by Theorem~\ref{ThmCanonform} the canonical form  of $(A,H)$ is of the form 
\begin{equation*}
\left(\bigoplus_{j=1}^t\left(J_{k_j}(\mu_j)\oplus J_{k_j}(\bar{\mu}_j)\right),\,\bigoplus_{j=1}^tQ_{2k_j}\right),
\end{equation*}
where $\mu_j$ is an $m$th root of $\lambda_j\in\mathbb{R}^-$. This means that there exists an invertible matrix $S$ such that 
\[S^{-1}AS=\bigoplus_{j=1}^t\left(J_{k_j}(\mu_j)\oplus J_{k_j}(\bar{\mu}_j)\right)\quad\textup{and}\quad S^*HS=\bigoplus_{j=1}^tQ_{2k_j}.
\]
Consider $S^{-1}BS=(S^{-1}AS)^m=\bigoplus_{j=1}^t\left(\left(J_{k_j}(\mu_j)\right)^m\oplus \left(J_{k_j}(\bar{\mu}_j)\right)^m\right)$, which is $S^*HS=\bigoplus_{j=1}^t Q_{2k_j}$-selfadjoint. By Lemma~\ref{LemOneBlockNegEig} it follows that the canonical form of $(B,H)$ is 
\[\left(\bigoplus_{j=1}^t\left(J_{k_j}(\lambda_j)\oplus J_{k_j}(\lambda_j)\right), \bigoplus_{j=1}^t\left(Q_{k_j}\oplus (-Q_{k_j})\right)\right).
\]

%But since $A^m=B$, the Jordan normal form of $B$ is then $J=\left(\bigoplus_{j=1}^t\left(J_{k_j}(\mu_j^m)\oplus J_{k_j}(\bar{\mu}_j^m\right)\right)$ and since each $\mu_j$ is an $m$th root of a negative real number, we have that $\mu_j^m=\bar{\mu}_j^m=\lambda_j$. From Theorem \ref{canonform}, the pair $(J,H_B)$ is in canonical form where $H_B=\bigoplus_{j=1}^t\left(\varepsilon_{1j} Q_{k_j}\oplus\varepsilon_{2j} Q_{k_j}\right)$ and by an appropriate change of Jordan basis we can take $\varepsilon_{1j}=1$ and $\varepsilon_{2j}=-1$ for all $j=1,\ldots,t$. Finally, the pair $(J,H_B)$ is the canonical form of $(B,H)$. 

Conversely, let $B$ be an $H$-selfadjoint matrix and assume that the canonical form of $(B,H)$ is as in \eqref{Jnegeiew} and \eqref{Hnegeiew}. We first consider the case of just one pair of blocks,
so assume 
\[
B=J_n(\lambda)\oplus J_n(\lambda)\quad\textup{and}\quad  H=Q_n\oplus (-Q_n),
\]
with $\lambda$ a negative real number. Let $\mu$ be an arbitrary $m$th root of $\lambda$, and let $\tilde{A}=J_n(\mu)\oplus J_n(\bar{\mu})$. Since $m$ is even and $\lambda<0$ we have that $\mu$ is nonreal; therefore the matrix $\tilde{A}$ is $Q_{2n}$-selfadjoint. Now, note that $\tilde{A}^m$ has $\lambda$ on the main diagonal and then Lemma~\ref{LemOneBlockNegEig} implies that the pairs $(\tilde{A}^m,Q_{2n})$ and $(B,H)$ are unitarily similar. Hence, there exists an invertible matrix $P$ such that the equations
\begin{equation*}
P^{-1}\tilde{A}^mP=B\;\;\textup{and}\;\; P^*Q_{2n}P=H
\end{equation*}
hold. From Lemma~\ref{S-1ASlemma} the matrix $A:=P^{-1}\tilde{A}P$ is an $H$-selfadjoint $m$th root of $B$. In the case where $B$ consists of more than one pair of blocks, the construction can be applied to each pair of blocks separately.
\end{proof}

For the case where $m$ is odd the following result holds.

\begin{theorem}\label{ThmNegEig-oddm}
Let $B$ be an $H$-selfadjoint matrix with a spectrum consisting only of negative real numbers. Then, for $m$ odd, there exists an $H$-selfadjoint matrix $A$ such that $A^m=B$.
\end{theorem}

\begin{proof}
Let $B$ be an $n\times n$ $H$-selfadjoint matrix with negative real eigenvalues. Assume that the pair $(B,H)$ is in canonical form and that $B$ consists of a single Jordan block, i.e.
\[B=J_{n}\left(\lambda\right)\quad\textup{and}\quad H=\varepsilon Q_{n},
\]
where $\varepsilon=\pm 1$ and $\lambda$ is a negative real number. Let $\mu$ be the real $m$th root of $\lambda$, and let $\tilde{A}=J_{n}(\mu)$. Then the Jordan normal form of $\tilde{A}^m$ is equal to $B$. Note that the matrix $\tilde{A}$ is $H_A$-selfadjoint where $H_A=\delta Q_{n}$, with $\delta=1$ or $\delta=-1$. Take $\delta=\varepsilon$.  Similarly as in Section~\ref{SecPosReal}, construct an invertible matrix $P$  such that the equations
\begin{equation*}
P^{-1}\tilde{A}^mP=B\;\;\textup{and}\;\; P^*H_AP=H
\end{equation*}
hold.  Finally, it follows from Lemma~\ref{S-1ASlemma} that the matrix $A:=P^{-1}\tilde{A}P$ is an $H$-selfadjoint $m$th root of $B$. In the case where $B$ consists of more than one block, the construction is applied to each block separately.
\end{proof}

We illustrate Theorem~\ref{ThmNegEig-oddm} with the following example.
\begin{example}
Let $m=5$ and let the matrices $B$ and $H$ be given by
\[B=\begin{bmatrix}
-1&1&0\\
0&-1&1\\
0&0&-1
\end{bmatrix}\quad\textup{and}\quad  H=Q_3=\begin{bmatrix}
0&0&1\\0&1&0\\1&0&0
\end{bmatrix}.\]
Then $(B,H)$ is in canonical form. The eigenvalues of any fifth root of the matrix $B$ are fifth roots of $-1$. 
%\[\mathcal{A}=\{e^{i\pi/5};\,e^{i3\pi/5};\,e^{i\pi}=-1;\,e^{-i3\pi/5};\,e^{-i\pi/5}\},
%\]  
Construct the following matrix by using the real fifth root of $-1$, i.e.\ let $\tilde{A}=J_3(-1)$. Note that $\tilde{A}$ is $H_A$-selfadjoint, where $H_A= Q_3$. Then 
\[\tilde{A}^5=\begin{bmatrix}
-1&5&-10\\0&-1&5\\0&0&-1
\end{bmatrix}.\]
Using the notation in Section~\ref{SecPosReal}, one can easily see that 
\begin{equation*}
p_1=\begin{bmatrix}
25y_3\\0\\0
\end{bmatrix},\;\;p_2=\begin{bmatrix}
5y_2-10y_3\\5y_3\\0
\end{bmatrix}\;\;\textup{and}\;\;p_3=y=\begin{bmatrix}
y_1\\y_2\\y_3
\end{bmatrix},
\end{equation*}
so we obtain the equations 
\begin{equation*}
1=\phi_1(y)=y^*H_A(\tilde{A}^5+I)^2y=25\bar{y}_3y_3,
\end{equation*}
\begin{equation*}
0=\phi_2(y)=y^*H_A(\tilde{A}^5+I)y=5\bar{y}_3y_2-10\bar{y}_3y_3+5\bar{y}_2y_3,
\end{equation*}
\begin{equation*}
0=\phi_3(y)=y^*H_Ay=\bar{y}_3y_1+\bar{y}_2y_2+\bar{y}_1y_3.
\end{equation*}
One solution of these equations is the real vector $y$ where $y_3=1/5$, $y_2=1/5$ and $y_1=-1/10$. Thus we have a matrix
\[P=\begin{bmatrix}
5&-1&\frac{-1}{10}\\0&1&\frac{1}{5}\\0&0&\frac{1}{5}
\end{bmatrix}\]
such that the equations $P^{-1}\tilde{A}^5P=B$ and $P^*H_AP=H$ hold. Finally, we note that the matrix 
\[A=P^{-1}\tilde{A}P=\begin{bmatrix}
-1&\frac{1}{5}&\frac{2}{25}\\0&-1&\frac{1}{5}\\ 0&0&-1
\end{bmatrix}\]
is an $H$-selfadjoint fifth root of $B$.\qed
\end{example}

%\begin{example}
%Let us now consider the $6\times 6$ matrix 
%\[B=\begin{bmatrix}
%-1&1&0\\0&-1&1\\0&0&-1
%\end{bmatrix}\oplus\begin{bmatrix}
%-1&1&0\\0&-1&1\\0&0&-1
%\end{bmatrix},\]
%which is $H$-selfadjoint, where $H=Q_3\oplus -Q_3$. Then we choose a conjugate pair from the fourth roots of $-1$ as given in \eqref{eqcalA4throots} $e^{i\pi/4}$ and $e^{-i\pi/4}$, and let
%\[\tilde{A}=\begin{bmatrix}
%e^{i\pi/4}&1&0\\0&e^{i\pi/4}&1\\0&0&e^{i\pi/4}
%\end{bmatrix}\oplus\begin{bmatrix}
%e^{-i\pi/4}&1&0\\0&e^{-i\pi/4}&1\\0&0&e^{-i\pi/4}
%\end{bmatrix}.\]
%This matrix is $H_A$-selfadjoint from Theorem~\ref{ThmCanonform}, where $H_A=Q_6$. Notice that $\tilde{A}^4$ has the Jordan normal form $B$. Therefore there exists an invertible matrix $P$ such that $P^{-1}\tilde{A}^4P=B$ and $P^*H_AP=H$. Consequently, the matrix $A=P^{-1}\tilde{A}P$ is a fourth root of $B$ and also $H$-selfadjoint.\qed
%\end{example}

\paragraph{Acknowledgements:}
This work is based on research supported in part by the DSI-NRF Centre of Excellence in Mathematical and Statistical Sciences (CoE-MaSS). Opinions expressed and conclusions arrived at are those of the authors and are not necessarily to be attributed to the CoE-MaSS.





\end{document}